\documentclass[11pt]{entcs}
\usepackage{entcsmacro,latexsym,amssymb,amsmath}
\def\.{{\cdot}}            
\def\<{\langle}            
\def\>{\rangle}            
\def\({\big(}              
\def\){\big)}              
\def\implies
{\hbox{$\Rightarrow$}}     
\def\lead{\leaders\hbox to 1.5ex{\hss${.}$\hss}\hfill}
\def\arr{\hbox to 60pt{\rightarrowfill}}
\def\larr{\hbox to 60pt{\leftarrowfill}}

\newskip\aline \newskip\halfaline
\def\1{{\bf1}}

\renewcommand{\Re}{{\mathbb R}}
\newcommand{\Nat}{{\mathbb N}}
\renewcommand{\P}{\mathcal{P}}
\newcommand{\C}{{\mathcal{C}}}

\newcommand{\ua}{{\uparrow}}
\newcommand{\da}{{\downarrow}}

\newcommand{\DCPO}{\mbox{\sf DCPO }}

 \newcommand{\DOM}{\mbox{\textsf{DOM} }} 
   
 \newcommand{\BCD}{\mbox{\textsf{BCD} }} 
 \newcommand{\RB}{\mbox{\textsf{RB} }}

\renewcommand{\O}{{\mathcal O}}

\newcommand{\Da}{\mathord{\mbox{\makebox[0pt][l]{\raisebox{-.4ex}
                           {$\downarrow$}}$\downarrow$}}}
\newcommand{\Ua}{\mathord{\mbox{\makebox[0pt][l]{\raisebox{.4ex}
                           {$\uparrow$}}$\uparrow$}}}
\renewcommand{\Pr}{\textsf{Prob}}

\renewcommand{\Pr}[1]{{\textsf{Prob}\,#1}}
\newcommand{\SPr}[1]{{\textsf{SProb}\,#1}}

\newcommand{\Rb}{\textsf{RB}}

\newcommand{\V}{{\mathbb V}}
\newcommand{\CT}{\mathbb{CT}}

\renewcommand{\to}{\longrightarrow}
\newcommand{\ran}{\text{ran}\,} 
\newcommand{\wid}[1]{\widehat{#1}}
\newcommand{\PI}{\wid{\pi}}

\def\firstheadline{}
\begin{document}
\begin{frontmatter}
  \title{Domains and Stochastic Processes}
\author{Michael Mislove}
\address{Department of Computer Science\\Tulane University, New Orleans, LA 70118}
\begin{abstract}
Domain theory has a long history of applications in theoretical computer science and mathematics. In this article, we explore the relation of domain theory to probability theory and stochastic processes. The goal is to establish a theory in which Polish spaces are replaced by domains, and measurable maps are replaced by Scott-continuous functions. We illustrate the approach by recasting one of the fundamental results of stochastic process theory -- Skorohod's Representation Theorem -- in domain-theoretic terms. We anticipate the domain-theoretic version of results like Skorohod's Theorem will improve our understanding of probabilistic choice in computational models, and help devise models of probabilistic  programming, with its focus on programming languages that support sampling from distributions where the results are applied to Bayesian reasoning. 
\end{abstract}
\end{frontmatter}
\section{Introduction}
The goal of this paper is establish a strong connection between domain theory and stochastic process theory. This follows the emergence over the past several years of random variables in domain theory as models for probabilistic choice (cf.~\cite{mislove-icalp,scott-stoch,barker}) in programming language semantics, and of probabilisitic programming semantics, an important tool for modeling programming languages that support sampling from probability distributions to study Bayesian inference (cf.~\cite{probprog1,probprog2}).  Our aim is to devise domain representations of stochastic processes that are amenable to computational analysis. 

Random variables are measurable maps $X\colon (S,\Sigma_S,\mu) \to (T,\Sigma_T)$ from a probability space $S$ to a measure space $T$. It is customary to identify a random variable $X$ with its law, $X_*\, \mu$, the push forward of the measure $\mu$ under $X$. A common setting is that of Polish spaces -- completely metrizable separable spaces, because the probability measures on a Polish space also are a Polish space in the weak topology. Our approach expands Polish spaces into domains, which allows us to represent each random variable as a Scott-continuous map. This allows us to use techniques and results from domain theory to prove analogs of results about stochastic processes. 

We illustrate our approach by reformulating Skorohod's Representation Theorem, a fundamental result in stochastic process theory, in the domain setting. 
In more detail,  Skorohod's Theorem~\cite{skorohod} states that any Borel probability measure on a Polish space $P$ is the law of a random variable $X\colon [0,1]\to P$. That is, if $\mu$ is a Borel probability measure on a Polish space $P$ and if $\lambda$ denotes Lebesgue measure on the unit interval, then there is a measurable map $X\colon [0,1]\to P$ satisfying $\mu = X_*\, \lambda$. Furthermore, if $\mu_n\to_w \mu$ in $\Pr{P}$ in the weak topology, then the random variables $X, X_n\colon [0,1]\to P$ can be chosen with laws $\mu$ and $\mu_n$, respectively, so that $X_n\to X$ almost surely wrt $\lambda$. This allows one to replace arguments about the weak convergence  of probability measures on a Polish space with arguments about almost sure convergence of measurable maps from the unit interval to the Polish space in question. 

In the domain approach, the unit interval equipped with Lebesgue measure is replaced by a suitable domain equipped with a corresponding probability measure, and the class of Polish spaces is replaced by a suitable category of domains.  In outline form, the approach relies on the following:
\begin{itemize}
\item It is well known that the Cantor set, $\C\simeq 2^\omega$, regarded as a countable product of two-point groups, is a \emph{standard probability space:} the canonical quotient map $\varphi\colon \C\to [0,1]$  is a Borel isomorphism taking Haar measure $\nu_\C$ on $\C$ to Lebesgue measure  (cf., e.g.,~\cite{brian-misl}).  
\end{itemize}
This implies there is a random variable $X\colon [0,1]\to P$ satisfying $X_*\,\lambda = \mu$ iff there is a random variable $X'\colon \C\to P$ with $X'_*\,\nu_\C = \mu$, for any probability measure $\mu$ on a Polish space $P$.
\begin{itemize}
\item We then expand the Cantor set into the Cantor tree: $\CT = \{0,1\}^\infty  = \{0,1\}^*\cup \{0,1\}^\omega$, the set of finite and infinite words over $\{0,1\}$.  This is a \emph{computational model} for $\C$, since $\CT$ is a countably based bounded complete domain when endowed with the prefix order that satisfies $\C\simeq \text{Max}\, \CT$. 
\item In fact, results in domain theory show every Polish space embeds as the space of maximal elements of some countably based bounded complete domain, $D_P$, and conversely, the space of maximal elements of any such domain is a Polish space in the relative Scott topology. 
\end{itemize}
Together, the previous results allow us to prove that for every random variable $X\colon [0,1]\to P$, there is a Scott-continuous map $f\colon\CT\to D_P$ having the same law as the random variable $X$: $f_*\, \nu_\C = (f\vert_\C)_*\,\nu_\C = X_*\,\lambda$.
\begin{itemize}
\item Just as every Polish space is the space of maximal elements of a countably based bounded complete domain, $D_P$, the family of Borel probability measures on a Polish space $P$ is the space of maximal elements of the domain of probability measures on the associated countably based bounded complete domain $D_P$. 
\item These associations are topological: the weak topology on the space $\textsf{Prob}\, \text{Max}\,D$ of  probability measures on a bounded complete domain $D$ coincides with the relative Lawson topology on the $\Pr D$, viewed as a domain (via the isomorphism with the valuations domain on $D$). In fact, the Lawson topology on $\Pr D$, for $D$ a countably based bounded complete domain, coincides with the weak topology. 
\end{itemize}
Combining these results leads to our main theorem:
\medbreak
\noindent\textbf{Theorem 1.} \emph{(Skorohod's Theorem for Bounded Complete Domains)\\ Let $D$ be a countably-based bounded complete domain, and let $\{\mu_n\}_n\in \textsf{Prob}\,D$ be a sequence of Borel probability measures satisfying $\lim_n \mu_n  = \mu \in \textsf{Prob}\,D$ in the Lawson topology (which agrees with the weak topology on probability measures). Then there are Scott-continuous maps $f,f_n\colon \CT\longrightarrow D$ satisfying $f_*\,  \nu_\C = \mu, f_{n*}\, \nu_\C = \mu_n$ for each $n$, and $f_n \longrightarrow f$ pointwise wrt to the Scott topologies. }

\begin{remark}
This result is somewhat weaker than Skorohod's Theorem for Polish spaces, since the convergence of the Scott-continuous functions $f_n$ is in the Scott topology. The family of functions $f_n$ described in the Theorem may not be directed, so we cannot conclude that $\lim f_n(x) = f(x)$ for all $x\in \CT$; we can only conclude that $\liminf_n  f_n(x) \geq f(x)$. Still, for those $x\in \CT$ for which $f(x)\in \text{Max}\, D$, we do get $\lim_n f_n(x) = f(x)$ in the Lawson topology. 
\end{remark}

Skorohod's Theorem is a corollary of Theorem 1 as follows. Any Polish space $P$ has a \emph{computational model:} a countably-based bounded complete domain $D_P$ for which $P$ is homeomorphic to the space Max$\,D_P$ of maximal elements endowed with the relative Scott topology. In fact, Max$\,D_P$ is a $G_\delta$, hence a Borel subset of $D_P$. So, if $\mu_n$ and $\mu$ are probability measures on $P$, then they are probability measures on $D_P$ concentrated on $\text{Max}\, D_P$, and so Theorem 1 implies they can be realized as laws for Scott-continuous maps $f, f_n\colon \CT\to D_P$, i.e., $f_*\,\nu_\C = \mu$ and $f_{n *}\, \nu_\C = \mu_n$ for each $n\geq 0$.
Since $\mu_n, \mu$ are concentrated on Max$\,D$, we can restrict $f_n,f$ to those points  $x\in \text{Max}\,\CT$ where $f_n(x), f(x)\in\text{Max}\, D$. It follows that $f_n(x)$ converges to $f(x)$ wrt the Lawson topology on $D$ for $\nu_\C$-almost all $x\in \text{Max}\,\CT$.

Finally, the canonical surjection $\phi\colon \C\to [0,1]$ preserves all sups and infs, so it has a lower adjoint $j\colon [0,1]\to \C$ that is a Borel isomorphism taking Lebesgue measure to $\nu_\C$. Thus $(f\circ j)_*\, \lambda = f_*\, j_*\lambda = f_*\, \nu_\C = \mu$ and similarly, $(f_n\circ j)_*\, \lambda = \mu_n$ are the random variables guaranteed by Skorohod's classic result. 

We also report the results of our research that concerns one of the longstanding problems in domain theory. The \emph{Jung-Tix Preoblem} asks whether there is a Cartesian closed category of domains for which the valuations monad is an endofunctor. We cannot answer this question, but we can contribute our knowledge of the structure of $\V\, D$ when $D$ is a complete chain:
\medbreak
\noindent\textbf{Theorem 2.} If $D$ is a countably based complete chain, then $\textsf{Prob}\, D$ and  $\textsf{SProb}\,D$ (the family of sub-probability measures on $D$) are continuous lattices. 
\medbreak
This result significantly expands our knowledge of the domain structure of the family of (sub)probability measures on a domain $D$. Indeed, up to this point, the only domains $D$ for which $\textsf{SProb}\,D$ is known to be a domain are: (i) a (rooted) tree, $T$, for which $\textsf{SProb}\,T\in$ \textsf{BCD}, the category of bounded complete domains, or (ii) a finite reverse tree $T^{rev}$, in which case $\textsf{SProb}\,T^{rev}$ is in \RB\cite{jungtix}.

The previous discussion hints at results that also are included in this research. While our interest in Skorohod's Theorem results in a focus on the probability measures, we also show that many of our results apply more broadly to the family of subprobability measures on a domain. Regarded as valuations, this is the \emph{probabilistic power domain}, a much-studied construct in domain theory. We explain the relationship between (sub-)probability measures and valuations in detail, and describe how the more general results concerning these measures follow along the same lines as the arguments we present for probability measures. In each case, the more general result can be obtained by the same proof strategy as the one for probability measures.

\subsection{Related Work}
Beginning in the mid-1990s, Abbas Edalat developed domain-theoretic approaches to a number of areas, including integration theory~\cite{edalat1}, stochastic processes~\cite{edalat2}, dynamical systems and fractals~\cite{edalat3}, and Brownian motion~\cite{bilokon}. The concept of a computational model emerged in Edalat's work on domain models of spaces arising in real analysis using the domain of compact subsets under reverse inclusion, where the target space arises as the set of maximal elements. The first paper formally presenting such a model was~\cite{edalat3}, where a domain model for locally compact second countable spaces was given. That paper presents a range of applications of the approach, including dynamical systems, iterated function systems and fractals, a computational model for classical measure theory  on locally compact spaces, and a computational generalization of Riemann integration.  Related work led to the formal ball model~\cite{edalat5} which was tailor-made for modeling metric spaces and Lipschitz functions. Further discussion of these developments occurs in our discussion of Polish spaces in Section~\ref{sec:polish} below.

Other related work concerns the development of random variable models of probabilistic computational processes. This began with~\cite{mislove-icalp}, a paper that provided a domain model for finite random variables. Further efforts saw limited success until a few years ago. The model proposed in~\cite{g-l-v-lics} turned out to be flawed, as was initially shown in~\cite{misl-anat1,misl-anat2}.   But inspired by ideas from~\cite{g-l-v-lics}, Barker~\cite{barker} devised a monad of random variables that gives an abstract model for randomized algorithms. This line of research was initiated by Scott~\cite{scott-stoch}, who showed how the $\P(\Nat)$ model of the lambda calculus could be extended naturally to support probabilistic choice with the aid of random variables $X\colon [0,1]\to \P(\Nat)$. Barker's results generalize Scott's approach by providing a model of randomized PCF that adds a version of probabilistic choice based on a random variables monad. Notably, this monad leaves important Cartesian closed categories of domains invariant -- in particular, the category \BCD of bounded complete domains, as well as the CCC \RB of retracts of bifinite domains invariant, and each enjoys a distributive law with respect to at least one of nondeterminism monads.\\[1ex] 

The rest of the paper is as follows. In the next section, we review the material we need from a number of areas, including domain theory, topology, and probability theory. Section 3 develops results about mappings from the Cantor tree to the space $\SPr{D}$ of sub-probability measures on a countably-based coherent domain $D$. Section 4 contains the main results of the paper, by first recalling the development of Polish spaces as computational models, and then presenting the main theorems.  Section 5 summarizes what's been proved, and discusses future work. 

\section{Background}\label{sec:background}
In this section we present the background material we need for our main results. 
\subsection{Domains}\label{subsec:domains}
Our results rely fundamentally on domain theory. Most of the results that we quote below can be found in \cite{abrjung} or \cite{comp}; we give specific references for those that appear elsewhere. 

To start, a \emph{poset} is a partially ordered set. A poset is \emph{directed complete} if each of its directed subsets has a least upper bound, where a subset $S$ is \emph{directed} if each  finite subset of $S$ has an upper bound in $S$. A directed complete partial order is called a \emph{dcpo}. 
The relevant maps between dcpos are the monotone maps that also preserve suprema of directed sets; these maps are usually called \emph{Scott continuous}. 

Restating things topologically, a subset $U\subseteq P$ of a poset is \emph{Scott open} if (i) $U = \ua U \equiv \{ x\in P\mid (\exists u\in U) \ u\leq x\}$ is an upper set, and (ii) if $\sup S\in U$ implies $S\cap U\not=\emptyset$ for each directed subset $S\subseteq P$. It is routine to show that the family of Scott-open sets forms a topology on any poset; this topology satisfies $\da x \equiv \{y\in P\mid y\leq x\} = \overline{\{x\}}$ is the closure of a point, so the Scott topology is always $T_0$, and it is $T_1$ iff $P$ is a flat poset. A mapping between dcpos is Scott continuous in the order-theoretic sense iff it is a monotone map that is continuous with respect to the Scott topologies on its domain and range.  We let \DCPO denote the category of dcpos and Scott-continuous maps; \DCPO is a Cartesian closed category.

If $P$ is a dcpo, and $x,y\in P$, then \emph{$x$ approximates $y$} iff for every directed set $S\subseteq P$, if $y\leq \sup S$, then there is some $s\in S$ with $x\leq s$. In this case, we write $x\ll y$ and we let $\Da y = \{x\in P\mid x\ll y\}$. A \emph{basis} for a poset $P$ is a family $B\subseteq P$ satisfying $\Da y\cap B$ is directed and $y = \sup (\Da y\cap B)$ for each $y\in P$. A \emph{continuous poset} is one that has a basis, and a dcpo $P$ is a \emph{domain} if $P$ is  continuous. An element $k\in P$ is \emph{compact} if $x\ll x$, and $P$ is \emph{algebraic} if $KP = \{ k\in P\mid k\ll k\}$ forms a basis. Domains are sober spaces in the Scott topology. 

Domains admit a Hausdorff refinement of the Scott topology which will play a role in our work. The \emph{weak lower topology} on $P$ has the sets of the form if $O = P\setminus \ua F$ as a basis, where $F\subset P$ is a finite subset. The \emph{Lawson topology} on a domain $P$ is the common refinement of the Scott- and weak lower topologies on $P$. This topology has the family 
$$\{ U\setminus\!\ua F\mid U\ \text{Scott open}\ \&\ F\subseteq P\ \text{finite}\}$$
as a basis. The Lawson topology on a domain is always Hausdorff, and a domain is \emph{coherent} if its Lawson topology is compact. We denote the closure of a subset $X\subseteq P$ of a domain in the Lawson topology by $\overline{X}^\Lambda$.

A result that plays an important role for us is the following:
\begin{lemma}\label{lem:cntbase}
Let $D$ be a countably based domain with countable basis $B_D$, and let $x\in D$. Then:
\begin{enumerate}
\item $x$ is the supremum of a countable chain $\{x_n\}_{n\in\Nat}$ with $x_n\ll x$ for each $n$. 
\item If $D$ is coherent, then there is a countable chain of Lawson-open sets $U_{n} = \{ \Ua x_n\setminus \ua F_n\mid x_n\ll x\not\in F_n\subseteq B_D\ \text{finite}\}$ with $x = \bigcap_{n} U_{n}$.
\end{enumerate}
\end{lemma}
\begin{proof}
(i):\ Since $D$ has a countable base, there is a countable directed set $B\subseteq \Da x$ with $\sqcup D = x$. If we enumerate $B = \{b_0,b_1,\ldots\}$, then we define the desired sequence $x_n$ as follows: $x_0 = b_0$, and if $x_0\ll x_1\ll \ldots\ll x_n$ have been chosen from $B$, then we choose $x_{n+1}\in B$ with $b_i\ll x_{n+1}$ for each $i\leq n$ and $x_n\ll x_{n+1}$. This extends the sequence, and then a standard maximality argument shows we can choose an countable sequence $x_n$ with $x_n\ll x_{n+1}\ll x$ for each $n$. Finally, $x = \sqcup_n b_n \leq \sqcup_n x_n$, since $b_n\ll x_{n+1}$ for each $n$, but $x_n\ll x$ for each $n$ implies $\sqcup_n x_n\leq x$. 

(ii): By definition of the Lawson topology, we know $\Ua x_n\setminus \ua F$ is an open set containing $x$ if $x\not\in F\subseteq B_D$, and there are countably many of these sets since $B_D$ is countable. It's then routine to extract a chain $U_n = \Ua x_n\setminus \ua F_n$ whose intersection also is $x$. 
\end{proof}

We also need some basic results about Galois adjunctions (cf.~Section 0-3 of \cite{comp}) in the context of complete lattices. If $L$ and $M$ are complete lattices, a \emph{Galois adjunction} is a pair of mappings $g\colon L\to M$ and $f\colon M\to L$ satisfying $f\circ g\leq 1_L$ and $g\circ f\geq 1_M$. In this case, $f$ is the \emph{lower adjoint}, and $g$ is the \emph{upper adjoint}. Lower adjoints preserve all suprema, and upper adjoints preserve all infima. In fact, each mapping $f$ between complete lattices that preserves all suprema is a lower adjoint; its upper adjoint $g$ is defined by $g(y) = \sup f^{-1}(\da y)$. Dually, each mapping $g$ preserving all infima is an upper adjoint; its lower adjoint $f$ is defined by $f(x) = \inf g^{-1}(\ua x)$. The cumulative distribution function of a probability measure on $[0,1]$ and its upper adjoint given in the introduction are examples we'll find relevant. 

Finally, we need some detailed information about two Cartesian closed categories of domains. 
We let \DOM denote that category of domains and Scott continuous maps; this is a full subcategory of \textsf{DCPO}, but it is not Cartesian closed. Nevertheless, \DOM has several Cartesian closed full subcategories. Two of particular interest to us are the full subcategory  \BCD of countably based bounded complete domains. Precisely, a \emph{bounded complete domain} is a domain in which  every non-empty subset  has a greatest lower bound. An equivalent statement is that every subset having an upper bound has a least upper bound. 

A second CCC of domains in which we are interested is \Rb,  the category of countably based domains that are retracts of bifinite domains, and Scott-continuous maps. The simplest way to make this definition precise is by saying that $D\in\RB$ iff there is a countable family $\{f_n\}_n$ of \emph{deflations} of $D$ satisfying $\text{Id}_D = \sup_{n} f_n$, where $f_n\colon D\to D$ is a deflation if $f_n$ is Scott continuous and $f_n(D)$ is finite. Moreover, since $f_n(x)\ll x$ for each deflation (cf.~\cite{comp} Lemma II-2.16), and $f_n(D)$ is finite for each $n$, the family $B_D = \bigcup_n f_n(D)$ is a countable 
basis for $D$. 

Finally, note that note that \BCD is a full subcategory of \Rb, and that \BCD and \RB consist of coherent domains.
%

\subsection{The probabilistic power domain}\label{sec:probpower}
A \emph{continuous valuation} on a domain $D$ is a mapping $\mu\colon \O(D)\to [0,1]$ from the family of Scott-open sets to the interval satisfying: 
\begin{itemize}
\item (Strictness) $\mu(\emptyset) = 0$,
\item (Modularity) $\mu(U\cup V) + \mu(U\cap V) = \mu(U) + \mu(V)$, for $U, V\in \O(D)$,
\item (Scott continuity) If $\{U_i\}_{i\in I}\subseteq \O(D)$ is directed, then $\mu(\bigcup_i U_i) = \sup_i \mu(U_i)$.
\end{itemize}
Valuations are ordered pointwise: $\mu\leq \nu$ iff $\mu(U)\leq \nu(U)$ for all $U\in\O(D)$. 
We denote the set of valuations over a domain $D$ with this order by $\V D$. This is often referred to as the \emph{probabilistic power domain} of $D$. We also denote by $\V_1 D$ the valuations $\mu$ satisfying $\mu(D) = 1$. As we will see, valuations correspond to subprobability measures, while members of $\V_1 D$ correspond to probability measures. 
The following is called the \emph{Splitting Lemma,} it is fundamental for understanding the domain structure of $\V D$ and of $\V_1 D$. 
\begin{theorem}\label{thm:split} (Splitting Lemma~\cite{jones}) 
Let $D$ be a domain and let $\mu = \sum_{x\in F} r_x\delta_x$ and $\nu = \sum_{y\in G} s_y\delta_y$ be \emph{simple valuations} on $D$. Then the following are equivalent:
\begin{enumerate}
\item $\mu \leq \nu \in \V D$,
\item There is a family $\{t_{x,y}\}_{\langle x,y\rangle\in F\times G}\subseteq [0,1]$ of \emph{transport numbers} satisfying:
\begin{itemize}
\item $r_x = \sum_{y\in G} t_{x,y}$ for each $x\in F$,
\item $\sum_{x\in F} t_{x,y} \leq s_y$ for each $y\in G$,
\item $t_{x,y} > 0\ \Rightarrow\ x\leq y$.
\end{itemize}
\end{enumerate}
Moreover, $\mu\ll \nu \in \V D$ iff  (i) $\sum_{x\in F} r_x < s_y$ for each $y\in G$ and (ii) $t_{x,y}$ satisfies $t_{x,y} > 0$ implies $x\ll y\in D$ for each $x\in F, y\in G$. 
\end{theorem}
This result can be used to show that, given a basis $B_D$ for $D$, the family $\{ \sum_{x\in F} r_x\delta_x\mid F\subseteq B_D, \sum_{x\in F} r_x < 1\}$ forms a basis for $\V D$; in particular, each sub-probability measure is the directed supremum of simple measures way-below it, so $\V D$ is a domain if $D$ is one. Moreover, Jung and Tix~\cite{jungtix} showed that $\V D$ is a coherent domain if $D$ is. 

Our interest is in countably-based coherent domains, in which case we can refine the Splitting Lemma~\ref{thm:split} and Lemma~\ref{lem:cntbase}.

\begin{notation}\label{nota:dyad}
We let $Dyad$ denote the \emph{positive} dyadic rationals in the unit interval. Likewise, $Dyad_n = \{ {k\over 2^n}\mid 0\leq k\leq 2^n\}$ is the set of dyadics with denominator $2^n$, for each $n\geq 1$.
\end{notation}

\begin{proposition}\label{prop:split}
Let $D$ be a coherent domain with countable basis $B_D$. Yhen:
\begin{enumerate}
\item  $\V\, D$ is a countably-based coherent domain with basis\\[1ex]
\centerline{ $B_{\V D} = \{ \sum_{x\in F} r_x\delta_x\mid F\subseteq B_D\ \text{finite}\ \&\ r_x\in Dyad\ \forall x\in F\}$.}\\[1ex] 
Moreover, 
if $\sum_{x\in F} r_x\delta_x\leq \sum_{y\in G} s_y\delta_y\in B_{\V D}$, then the family $\{t_{x,y}\}_{(x,y)\in F\times G}$ of transport numbers from the Splitting Lemma~\ref{thm:split} satisfy $t_{x,y}\in Dyad$ for all $(x,y)\in F\times G$.
\item The family $\V_1 D = \{\mu\in \V D\mid \mu(D) = 1\}$ is a countably-based coherent domain
with basis\\
\centerline{$B_{\V_1 D} = \{ \sum_{x\in F} r_x\delta_x\mid \perp\, \in F\subseteq B_D\ \text{finite}, \sum_x r_x = 1\ \&\ r_x\in Dyad\ \forall x\in F\}$.}\\[1ex]
Moreover, $\sum_{x\in F} r_x\delta_x \ll (\text{resp.,}\, \leq)  \sum_{y\in G} s_y\delta_y$ in $\V_1D$ iff the transport numbers $\{t_{x,y}\}$ from~\ref{thm:split} satisfy:
\begin{itemize}
\item $r_x = \sum_y t_{x,y}$ for each $x\in F$, and
\item $\sum_{x} t_{x,y} < (\text{resp.,}\, \leq)\, s_y$ if $y\not=\perp_D$ for each $y\in Y$.
\end{itemize}
In particular, $r_\perp > (\text{resp.,}\, \geq)\, s_\perp$. 
\end{enumerate}
Finally, each $\mu\in \V D$ (respectively, $\V_1 D$) is the supremum of a countable chain $\mu_n\in B_{\V D}$ (respectively, $\V_1 D$). 
\end{proposition}
\begin{proof}
(i): It is shown in~\cite{jungtix} that $\V D$ is coherent if $D$ is, and the Splitting Lemma~\ref{thm:split} implies $B_{\V D}$ is a basis for $\V D$. 

We next outline the proof of the second point -- that the transport numbers $t_{x,y}$ between comparable simple measures all belong to $Dyad$ if the coefficients of the measures do. This follows from the proof of the Splitting Lemma~\ref{thm:split} as presented in~\cite{jones}: That proof is an application of the Max Flow -- Min Cut Theorem to the directed graph $G = (E,N)$ which has a ``source node," $\perp$, connected by an outgoing edge of weight $r_x$ to each ``node" $x\in F$, a  ``sink node," $\top$, with an incoming edge of weight $s_y$ from each element $y\in G$,  and edges from $x\in F$ to $y\in G$ of large weight (say, $1$), if $x\leq y$. 

A \emph{flow} is an assignment $f\colon E \to \Re_+$ of non-negative numbers to each edge so that $f(u,v) \leq c(u,v)$ for nodes $u, v$, where $c(u,v)$ is the weight as defined above, and satisfying $\sum_u f(u,v) = \sum_t f(v,t)$ for each node $t\not=\perp, \top$. The \emph{value} of a flow $f$ is $val f = \sum_{u} f(\perp\! u)$, the total amount of flow out of $\perp$ using $f$.  A \emph{cut} is a partition of $N = S\stackrel{\cdot}{\cup} T$ with $\perp\in S$ and $\top\in T$. The value of a flow across the cut $T$ is $\sum_{(u,v) \in S\times T\,\cap\, E} f(u,v)$. 

The Max Flow--Min Cut Theorem asserts that the maximum flow on a directed graph is equal to the minimum cut. It is proved by applying the Ford--Fulkerson Algorithm~\cite{bolla}. The algorithm starts by assigning the minimum flow $f(u,v) = 0$ for all edges $(u,v)\in E$, and then iterates a process of selecting a path from $\perp$ to $\top$, calculating the residual capacity of each edge in the path, defining a residual graph $G_f$, augmenting the paths in $G_f$ to include additional flow, and then iterating. The result of the algorithm is the set of flows along edges across the cut, which are the transport numbers $t_{x,y}$ in our case. Since the calculations of new edge weights involve only arithmetic operations, and since the dyadic rationals form a subsemigroup of $\Re_+$, the resulting transport numbers $t_{x,y}$ are dyadic rationals if the coefficients of the input distributions are dyadic. Moreover, if all the weights are rational, then the algorithm halts producing the maximum flow across the network. 

(ii): Let $\varphi\colon \V D \to\V_1 D$ by $\varphi(\mu) = \mu + (1-\mu(D))\delta_\perp$. 
Then $\varphi$ is Scott continuous, since $\sup_n \mu_n = \mu$ implies $\sup_n \mu_n(D) = \mu(D)$. $\varphi$ also is a \emph{projection:} $\varphi\circ \varphi = \varphi$. So, Theorem I-1.22 of~\cite{comp} implies $\V_1 D = \varphi(\V D)$ is continuous, and $\mu\ll \nu\in \V_1 D$ iff $(\exists \mu'\in \V D)\ \mu\leq \varphi(\mu')\ \&\ \mu' \ll_{\V D} \nu$. This implies $B_{\V_1D}$ is a basis for $\V_1 D$, since $B_{\V D}$ is a basis is one for $\V D$ by part (i). The final claim follows from these results and the characterization of the transport numbers $\{t_{x,y}\}$ from Theorem~\ref{thm:split}. 
\end{proof}

\subsection{The Jung-Tix Problem and the special case of chains}\label{subsec:chain}
We now turn our attention to a longstanding problem in domain theory. The \emph{Jung-Tix Problem} asks whether there are any Cartesian closed categories of domains for which the valuations monad $\V$ is an endofunctor. We do not have an answer, but we can offer insight to the question. We are able to show that the probabilistic power domain of any complete chain is a continuous lattice. We include the result here because we discovered a proof of this result during our research on how to express Skorohod's Theorem in domain-theoretic terms. 

\begin{notation}\label{nota:chain}
Throughout this section, \textbf{we assume $D$ is a chain.}  
\end{notation}
 
 \begin{definition}\label{def:cmd}
 Let $\mu$ be a sub-probability measure on $D$. The \emph{cumulative distribution function} $F_\mu\colon D\to [0,1]$ is defined by $F_\mu(x) = \mu(\da x)$.  
 \end{definition}
 
 \begin{proposition}\label{prop:cmd}
 For each $\mu\in \SPr{D}$, $F_\mu$ preserves all infima. 
 \end{proposition}
 \begin{proof}
 Let $\mu$ be a sub-probability measure. If $x\leq y\in D$, then $\da x\subseteq \da y$, so $F_\mu(x) = \mu(\da x) \leq \mu(\da y) = F_\mu(y)$. So $F_\mu$ is monotone, and since $D$ is a chain, this means $F_\mu$ also preserves finite infima. Now, any filtered set $A\subseteq D$ is totally ordered  because $D$ is. Then $\da \inf A = \bigcap_{x\in A} \da x$, and so\\[1ex]
 \centerline{$F_\mu(\da \inf A) = F_\mu(\bigcap_{x\in A} \da x) = \mu(\bigcap_{x\in A} \da x) = \inf_{x\in A} \mu(\da x) = \inf_{x\in A} F_\mu(x)$,}\\[1ex]
 where the next-to-last equality follows from the fact that, being a  Scott-continuous valuation on $D$, $\mu$ preserves directed unions of open sets, so it preserves filtered intersections of closed sets, such as $\{ \da x\mid x\in A\}$. This shows $F_\mu$ also preserves filtered infima, and so it preserves all infima.
 \end{proof}
 
Since $F_\mu$ preserves all infima and $D$ is a continuous lattice, it follows that $F_\mu$ is an upper adjoint, so it has a unique lower adjoint $G_\mu \colon [0,1]\to D$ defined by $G_\mu(r) = \inf F_\mu^{-1}(\ua r)$. We denote this relationship by $F_\mu \dashv G_\mu$. 

We recall some facts about such adjoint pairs; for more detail, see Chapter 0 of~\cite{comp}. 
First, each component of an adjoint pair $f\colon L \to M$, $g\colon M\to L$ with $f\dashv g$ determines the other. The formula for $G$ above shows how to define the lower adjoint, given an upper adjoint: $g(x) = \inf f^{-1}(\ua x)$. Conversely, given a lower adjoint $g$, the upper adjoint $f$ is given by $f(y) = \sup g^{-1}(\da y)$. Upper adjoints preserve all infima, and lower adjoints preserve all suprema. Moreover, if $f\dashv g$ and $f'\dashv g'$, then $f\leq f'$ iff $g'\geq g$. Finally, the components earn their names because of the relationship $f\circ g\geq 1_M$ and $g\circ f\leq 1_L$.
 
\begin{proposition}\label{prop:uuperadj}
 If $\mu$ is a sub-probability measure on $D$ with cumulative distribution function $F_\mu$, then the upper adjoint, $G_\mu\colon [0,1]\to D$ satisfies $G_\mu\, \lambda = \mu$, where $\lambda$ denotes Lebesgue measure. 
\end{proposition}
\begin{proof}
If $x\in D$, then 
\begin{eqnarray*}
G_\mu\, \lambda(\da x) & = & \lambda(G_\mu^{-1}(\da x))\\
& = & \lambda(\da F_\mu(x)) \qquad\qquad \qquad F_\mu\dashv G_\mu\\
& = & F_\mu(x) = \mu(\da x)
\end{eqnarray*}
Since $G_\mu\,\lambda$ and $\mu$ agree on Scott-closed sets, it follows that $G_\mu\, \lambda = \mu.$
\end{proof} 
 
 \begin{theorem}\label{thm:chain}
 If $D$ is a chain and $KD = \{\perp\}$, then $G\mapsto G\, \lambda\colon [[0,1]\to D]\to \SPr{D}$ is an order-isomorphism. Therefore, $\SPr{D}$ is a continuous lattice, and the same is true of $\Pr{D}$ is a domain.
 \end{theorem}
 \begin{proof}
 Each Scott-continuous map $G\colon [0,1]\to D$ preserves all suprema, since the domain $D$ is a chain. And each such map determines a sub-probability measure $G\, \lambda$. Then the cumulative distribution $F_{G\,\lambda}\colon D\to [0,1]$ satisfies $F_{G\,\lambda}(x) = G\,\lambda(\da x) = \lambda(G^{-1}(\da x)) = \sup G^{-1}(\da x)$. This means $F_{G\,\lambda}$ is the upper adjoint of $G$. Since upper and lower adjoints uniquely determine one another, the mapping $G\mapsto G\,\lambda$ has an inverse sending $\mu$ to the lower adjoint of $F_\mu$. 
 
 For the order structure, suppose $G\leq G'\in [[0,1]\to D]$. We show $G\,\lambda\leq G'\,\lambda$: Then given $x\in D$ and $r\in [0,1]$, if $G'(r) \leq x$, then $G(r)\leq x$; said another way, $G'^{-1}(\da x) \subseteq G^{-1}(\da x)$, so \\[1ex]
 \centerline{$G'\,\lambda(\da x) = \lambda(G'^{-1}(\da x)) = \sup G'^{-1}(\da x) \leq \sup G^{-1}(\da x) = \lambda(G^{-1}(\da x)) = G\,\lambda(\da x)$.}\\[1ex]
If $x =\, \perp$, $G\,\lambda(\Ua x) = G\,\lambda (D) \leq G'\,\lambda(D) = G'\,\lambda(\Ua x)$. On the other hand, since $KD = \{\perp\}$, then $x > \perp$ implies  $D = \da x\stackrel{\cdot}{\cup} \Ua x$, so we have\\[1ex]
 \centerline{$G\,\lambda(\Ua x) = G\,\lambda(D) - G\,\lambda(\da x) \leq G\lambda(D) - G'\,\lambda(\da x)\leq G'\lambda(D) - G'\lambda(\Ua x)  = G'\,\lambda(\Ua x)$.}\\[1ex] 
Since $D$ is a chain, every Scott-open set has the form $\Ua x$ for some $x\in D$, so $G\,\lambda\leq G'\,\lambda$. 

Conversely, if $\mu\leq \nu$, then $\mu(\da x)\geq \nu(\da x)$ by the same argument we used above,  so $F_\mu(x) = \sup \mu(\da x) \geq \sup  \nu(\da x) = F_\nu(x)$. It follows that $G_\mu\leq G_\nu$ from our remarks about adjoint pairs. 

Thus, the correspondence $G\mapsto G\,\lambda$ is an order-isomorphism. Since $[0,1]$ and $D$ are continuous lattices, they are both bounded complete domains, so $[[0,1]\to D]$ is a bounded complete domain. But $x\mapsto \top$ is the largest element of $[[0,1]\to D]$, so this is a continuous lattice.   It follows that $\SPr{D}$ is a continuous lattice as well. 

For the final claim, the mapping $\mu\mapsto \mu + (1 - \mu(D))\delta_\perp: \SPr{D}\to \Pr{D}$ is a closure operator that preserves directed sups, and the image of a continuous lattice under such a closure operator is a continuous lattice (cf.~\cite{comp}, Definition 0-2.10ff.).
 \end{proof}
 
\subsection{Valuations versus sub-probability measures}
It is straightforward to show that each Borel sub-probability measure on a domain $D$ restricts to a Scott-continuous valuation on the Scott-open sets of $D$. The converse, that each Scott-continuous valuation on a dcpo extends to a unique Borel sub-probability measure was shown by Alvarez-Manilla, Edalat and Saheb-Djorhomi~\cite{alvman}. 

The next step is to link the order-structure of $\V D$ to the family \textsf{SProb}$\, D$ of sub-probability measures on $D$, and this requires the next result. We recall that a \emph{simple sub-probability measure} on a space $X$ is a finite convex sum $\sum_{x\in F} r_x\delta_x$, where $F\subseteq X$ is finite, $r_x\geq 0$ for each $x\in F$, and $\sum_{x\in F} r_x\leq 1$. 
We also
recall that the real numbers, $\Re$, are a continuous poset whose Scott topology has the intervals $(a,\infty)$ as a basis, and whose Lawson topology is the usual topology. 
\begin{proposition}\label{weak-law}
Let $D$ be a coherent domain, and let $\mu, \nu$ be sub-probability measures on $D$. Then the following conditions are equivalent:
\begin{enumerate}
\item $\mu\leq \nu\in \V D$. 
\item For each Scott-continuous map $f\colon D\to \Re_+$, $\int f d\mu \leq \int f d\nu$.
\item For each monotone Lawson-continuous $f\colon D\to \Re_+$, $\int f\, d\mu \leq \int f\, d\nu$.
\end{enumerate}
\end{proposition}
\begin{proof}
We show the result for simple measures, which then implies it holds for all measures since $\V D$ is a domain -- so its partial order is (topologically) closed -- in which the simple measures are dense.
  
So, suppose $\mu = \sum_{x\in F} r_x\delta_x $ and $\nu = \sum_{y\in G} s_y \delta_y$ are simple measures on $D$. 
\medbreak
\noindent \textbf{(i) implies (ii):} Suppose that $\mu \leq \nu \in \V D$. If $f\colon D\to \Re_+$, then $\int f d\mu = \sum_{x\in F} r_x\cdot f(x)$ and $\int f d\nu = \sum_{y\in G} s_y\cdot f(y)$. Since $\mu\leq \nu$, there are $t_{x,y}\in [0,1]$ guaranteed by the Splitting Lemma~\ref{thm:split}, and so
\begin{eqnarray*}
\int f d\mu & = & \sum_{x\in F} r_x\cdot f(x) = \sum_{x\in F}\sum_{y\in G} t_{x,y}\cdot f(x)\\
& \leq & \sum_{x\in F}\sum_{y\in G} t_{x,y}\cdot f(y) \leq \sum_{y\in G} s_y\cdot f(y) = \int f d\nu,
\end{eqnarray*}
where the first inequality follows from the facts that $t_{x,y} > 0$ implies $x\leq y$ and $f$ is monotone. This shows (i) implies (ii).
\medbreak
\noindent \textbf{(ii) implies (iii):} Since monotone Lawson continuous maps are Scott continuous, this is obvious.
\medbreak
\noindent \textbf{(iii) implies (i):} Let $U$ be a Scott-open subset of $D$, and let $H = (F\cup G)\setminus U$. Using the facts that $D$ is coherent, so its Lawson topology is compact Hausdorff, and that $H$ is finite, we define a family $\{ U_d\mid d\in Dyad\}$ of Scott-open upper sets indexed by $Dyad$, the dyadic numbers in $[0,1]$ as follows: We let $U_0 = D\setminus \da H, U_1 = U$, and for $d < d'$, we recursively choose $U_d\supseteq \overline{U_{d'}}^\Lambda$, the Lawson-closure of $U_{d'}$. Then define a mapping\\[1ex]
\centerline{$f\colon D\to [0,1]$ by $f(x) = 0$ if $x\in \da H$, and otherwise $f(x) = \inf \{ d\mid x\in U_d\}$.}\\[1ex] 
Since the family $\{U_d\}$ consists of Scott-open sets satisfying $U_d\supseteq \overline{U_{d'}}^\Lambda$ for $d < d'$, this mapping is monotone, and the standard Urysohn Lemma argument (cf.\ Theorem 33.1~\cite{munkres}) shows it is Lawson continuous. So, $\int f d\mu\leq \int f d\nu$ by assumption. 

Since $\mu$ and $\nu$ are simple, $\int f d\mu = \sum_{x\in F\setminus H} r_x\cdot f(x)$ and $\int f d\nu = \sum_{y\in G\setminus H} s_y\cdot f(y)$. By construction, $(F\cup G)\setminus H \subseteq U = U_1$, so\\[1ex] 
$\sum_{x\in F\setminus H} r_x\cdot f(x) = \sum_{x\in F\setminus H} r_x = \mu(U)$, and
$\sum_{y\in G\setminus H} s_y\cdot f(y) = \sum_{y\in G\setminus H} s_y = \nu(U)$,\\[1ex]  and so 
$\mu(U) = \int f d\mu \leq \int f d\nu = \nu(U)$, as required.
\end{proof}

Proposition~\ref{weak-law} tells us we can realize the domain order structure of $\V D$ on \textsf{SProb}$\, D$ using the classical approach of integration against functions $f\colon D\to \Re_+$. Put another way, the mapping $\psi\colon \V D\to \textsf{SProb}\, D$ that realizes a valuation as a sub-probability measure is not only a bijection, but an order isomorphism if one equips \textsf{SProb}$\, D$ with the order described in the Proposition. It should be noted the same is true for \textsf{Prob}$\, D$, with essentially the same proof. Moreover, the isomorphisms are also homeomorphisms:

\begin{theorem}\label{thm:weak=laws}~\cite{weaktop,vanbreug}
If $D$ is a coherent domain, then the Lawson topology on $\textsf{SProb}\, D$ is the weak topology. The same holds for $\textsf{Prob}\, D$. 
\end{theorem}

\begin{remark}
Thus, Theorem~\ref{thm:weak=laws} says we can regard $\V D$ and \text{SProb}$\, D$ as one and the same, from a domain-theoretic perspective,  and the same holds for $\V_1D$ and \textsf{Prob}$\, D$. We sometimes will ``confuse" these two views of valuations / sub-probability measures without explicitly noting the identification. 
\end{remark}

\section{Domain Mappings from the Cantor Tree}\label{sec:cantor}
The \emph{Cantor tree} is the family $\CT = \{0,1\}^*\cup \{0,1\}^\omega$ of finite and infinite words over $\{0,1\}$ in the prefix order. Equivalently, $\CT$ is the full rooted binary tree which is directed complete, and since it is countably based, this means every directed supremum can be achieved as the supremum of an increasing countable chain. $\CT$ will play the role of the unit interval in our approach to generalizing Skorohod's Theorem to the domain setting. For that, we need some preliminary definitions.  

 An \emph{antichain} is a non-empty subset $A\subseteq \CT$ satisfying $a,b \in A$ implies $a$ and $b$ do not compare in the prefix order. 
%

\begin{notation}\label{not:nota1}
We establish some notation for what follows:
\begin{enumerate}
\item We let $\C_n \simeq 2^n$ be the set of $n$-bit words in $\CT$, which forms an antichain. Recall that there is a well-defined retraction mapping $\pi_n\colon \C\to \C_n$ from the Cantor set onto $\C_n$ sending each infinite binary word to its $n$-bit prefix. In addition, if $m \leq n$, then there is a map $\pi_{m,n}\colon \C_n \to \C_m$ that sends each $n$-bit word to its $m$-bit prefix. 
\item Since $\C_n$ is finite, the set $\da \C_n\subseteq \CT$ is Scott closed. Then both $\pi_n$ and $\pi_{m,n}$ extend to mappings $\PI_n\colon \CT\to \da \C_n$ and $\wid{\pi_{m,n}}\colon \da\C_n\to \da\C_m$ that send each element of $\CT$ to its largest prefix in $\da C_n$, respectively its largest prefix in $\C_m$.   Note that $\PI_m = \wid{\pi_{m,n}}\circ\PI_m$ if $m\leq n$.

\item The projection $\pi_n$ has a corresponding embedding $\iota_n\colon \C_n\to \C$ sending an $n$-bit word to the infinite word all of whose coordinates $m> n$ are $0$. Then $\pi_n\circ \iota_n = 1_{\C_n}$ and $ \iota_n\circ\pi_n \leq 1_\C$ form an embedding-projection pair, where we order $\C\simeq 2^\omega$ in the lexicographic order.  
%

\item The set $\C_n$  of $n$-bit words also can be given the lexicographic order. Then each dyadic rational $r\in [0,1]$ that can be expressed as $r = {k_r\over 2^n}$, the interval in $C_n$ from $0$ to $r$, or, equivalently, the first $k_r$ $n$-bit words. 

Moreover, each sequence of such dyadics, $r_1,\ldots, r_k$ whose sum is at most $1$ can be expressed as successive intervals, $r_1 = [0,\ldots, k_1], r_2 = [k_1+1,\ldots, k_1+k_2]$, etc. 
\end{enumerate}
\end{notation}

The proof of the next result relies on a version of Hall's Marriage Problem~\cite{hall}. The original version concerns a bipartite graph $G = (X + Y, E)$, where $X$ and $Y$ are the disjoint sets of nodes and all edges in $E$ connect a node of $X$ to one of $Y$. A \emph{matching} is a subset $M\subseteq E$ satisfying each node of $X$ has at most one edge in $M$, and likewise for $Y$, and no two edges in $M$ share any common nodes. A \emph{perfect matching} is one where every node of $X$ and every node of $Y$ has an edge in $M$. Hall's Marriage Problem states that there is a perfect matching iff, for each subset $S\subseteq X$ there are at least $|S|$ edges in $E$ from some node of $S$ to some node of $Y$. 

The generalization we need is for the case of matching each node of $X$ to $k$ nodes in $Y$ so that no two edges in $M$ share any nodes of $Y$ (and so each node of $X$ has edges to $k$ distinct nodes in $Y$ none of which is shared with any other node of $X$). The generalization of Hall's Marriage Problem states that such a $k$ matching exists -- i.e.,  there is a subset $M\subseteq E$ satisfying each node of $X$ has at least $k$ edges in $M$, and each node of $Y$ has at most one edge in $M$ -- iff, for each subset $S\subseteq X$, there are at least $k\cdot|S|$ edges in $E$ connecting some node of $S$ to a node of $Y$. The generalization follows from the original version by first making $k$ copies of each node of $X$, duplicating the edges in $E$ for each of these new nodes, applying the original version, and then collapsing the resulting perfect matching back to the original graph $G$. 

The following is key to our results:

\begin{proposition}\label{prop:key}
Let $D$ be a domain and let $\mu = \sum_{x\in F} r_x \delta_x\leq  \sum_{y\in G} s_y\delta_y =\mu'$ be simple probability measures on $D$ with $r_x, s_y\in Dyad$ for every $x$ and $y$. 
Suppose further that $f_m\colon \C_m\to D$ satisfies $f_{m*}\, \nu_m = \mu$, where $\nu_m$ is normalized counting measure on $\C_m$. 
Then there are $n > m$ and $f_n\colon \C_n\to D$ satisfying $f_{n*}\, \nu_n = \mu'$ and $f_m\circ  \pi_{m,n}\leq f_n$.
\end{proposition}

\begin{proof}
Since $\mu\leq \mu'$ are simple probability measures, the Splitting Lemma (Theorem~\ref{thm:split}) implies there are transport numbers $\{t_{x,y}\}$ satisfying $r_x = \sum_y t_{x,y}$ and $\sum_x r_x = s_y$ for all $x,y$. Moreover, since $r_x,s_y\in Dyad$ for all $x,y$, the transport numbers $t_{x,y}\in Dyad$ as well, as a proof similar to that for Proposition~\ref{prop:split} shows. We then choose $n > m$ such that $r_x, s_y, t_{x,y}\in Dyad_n$, for all $x,y$, where we recall $Dyad_n = \{{k\over 2^n}\mid 0\leq k\leq 2^n\}$ (cf.~\ref{nota:dyad}). Then\\ 
\centerline{$\mu = \sum_{x\in F} r_x\delta_x$ with $r_x = {k_x\over 2^n}$, so $\mu = \sum_{x\in F}  {k_x\over 2^n} \delta_x$}\\
and\\ \centerline{$\mu' = \sum_{y\in G} s_y\delta_y$ with $s_y = {k_y\over 2^n}$, so $\mu' = \sum_{y\in G}  {k_y\over 2^n} \delta_y$.}\\ 
Since $t_{x,y} \in Dyad_n$ for each $x,y$, it follows that, for each $x\in F$, the family $t_{x,y}$ distributes the mass $r_x = {k_x\over 2^n} $ associated to $\delta_x$ in $\mu$ to ${k_x\over 2^n}$ of the mass in $\mu'$. 

Since $f_m\colon \C_m\to D$ satisfies $f_{m*}\,\nu_m=\mu$, 
we also have $\mu = \sum_{x\in F} {k'_x\over 2^m} \delta_{x}$, and so 
${k'_x\over 2^m} = {k_x\over 2^n}$ for each $x\in F$. That is, $2^{n-m}k'_x = k_x$ for each $x\in F$. 
So, each $i\in \C_m$ is sent via $i\stackrel{f_{m*}}{\longrightarrow}  {1\over 2^m}\delta_{f(i)}\stackrel{\{t_{x,y}\}_y}{\longrightarrow} \mu'$ to ${2^{n-m}\over 2^n}$ of the mass of $\mu'$. 

We define a bipartite graph with node sets $X, Y$ where $X = \C_m$ and 
$Y = \bigcup_{y\in G} \{ (i,\delta_y)\mid 0<i\leq k_y\}$, and whose edges are\\ 
\centerline{$E = \{(i,(j,\delta_y))\mid t_{f(i),y}\ \text{sends mass at }\ \delta_{f(i)}\ \text{to } (j,\delta_y)\}$.}\\
Then each node $i\in X =\C_m$ has $2^{n-m}$ edges incident to it in $E$, and each node $(j,\delta_y)\in Y$ has exactly one edge incident to it. By construction, for each $S\subseteq X$, there are $2^{n-m}\cdot |S|$ edges in $E$ from some node in $S$ to some node in $Y$. The generalization of Hall's Marriage Problem described above then implies there is a $2^{n-m}$-matching, and since $\mu, \mu'$ are probability measures, $s_y = \sum_x t_{x,y}$ for each $y\in G$, so there is a total function $\rho\colon Y\twoheadrightarrow X$ satisfying $|\rho^{-1}(i)| = 2^{n-m}$ for each $i\in X$. 

On the other hand, the projection $\pi_{m,n}\colon \C_n\to \C_m$ sends $2^{n-m}$ $n$-bit words to each element of $\C_m$, so for each $i\in \C_m$ there is a bijection $b_i\colon \pi^{-1}_{m,n}(i)\to \rho^{-1}(i)$. Taken together, these $b_i$s define a map $g_n\colon C_n\to Y$, and if we let $p_Y\colon Y \to D$ by $p_Y(j,\delta_y) = y$ and define $f_n = p_Y\circ g_n$, then $f_m\circ \pi_{n,m} \leq f_n$ by construction. That $f_{n*}\, \nu_n = \mu'$ follows from the fact that $|f_n^{-1}(y)| = {k_y\over 2^n} = s_y$ for each $y\in G$, again by construction. 
\end{proof}

For our next result, we need some information about the weak topology on $\textsf{SProb}\, D$. The result we need follows from the Portmanteau Theorem~\ref{thm:port}~(cf., e.g.,~\cite{billings}), and a proof can be found as Corollaries 15 and 16 in~\cite{vanbreug}:

\begin{theorem}\label{thm:topvsweak} Let $D$ be a countably based coherent domain endowed with the Borel $\sigma$-algebra. Then the weak topology on $\textsf{SProb}\,D$ is the same as the Lawson topology on $\textsf{SProb}\,D$ when viewed as a family of valuations. 

Moreover, for a family $\mu_n,\mu\in\textsf{SProb}\,D$, the following are equivalent:
\begin{enumerate}
\item $\mu_n\to_w\mu$
\item Both of the following hold:
\begin{itemize}
\item $\lim\sup_n \mu_n(E)\leq \mu(E)$ for all finitely generated upper sets $E\subseteq D$.
\item $\lim\inf_n \mu_n(O)\geq \mu(O)$ for all Scott-open sets $O\subseteq D$.
\end{itemize}
\item $\lim\inf_n \mu_n(O)\geq \mu(O)$ for all Lawson-open sets $O\subseteq D$.
\end{enumerate}
\end{theorem}
\begin{corollary}\label{cor:topvsweak}
If $D$ is a countably based coherent domain, then the weak topology on $\textsf{Prob}\, D$ coincides with the Lawson topology. 
\end{corollary}
\begin{proof} 
The weak topology on \textsf{Prob}$\, D$ is compact Hausdorff since the Lawson topology on $D$ is compact Hausdorff. Then the same argument used in the proof of Corollaries 15 and 16 of~\cite{vanbreug} shows that the weak topology on \textsf{Prob}$\, D$ is finer than the Lawson topology. Since the latter is Hausdorff, the topologies coincide. 
\end{proof}

For the next result, we identify  $\C$ with the Cantor set of maximal elements in $\CT$, the Cantor tree, and recall that $\nu_C$ denotes Haar measure on $\C$ viewed as a countable product of two-point groups.

\begin{theorem}\label{thm:mapping}
Let $D$ be an \RB domain with countable basis, $B_D = \bigcup_k d_k(D)$, where $\cdots d_k\leq d_{k+1}\cdots$ is a countable sequence of deflations on $D$ with $\sup_k d_k = Id_D$ (see the discussion of \RB preceding Subsection~\ref{sec:probpower}). 
\begin{enumerate}
\item  If $\mu \in \textsf{Prob}(D)$, then there is a Scott-continuous map $X\colon \CT\to D$\footnote{Because of the many functions involved in the proof, we revert to the probability theory approach of denoting random variables by capital letters.} satisfying $X_*\, \nu_C = \mu$. 
\item Furthermore, if $\mu_n\to_w\mu$ are probability measures on $D$ converging to $\mu$ in the weak topology, then there are Scott-continuous maps $X_n\colon\CT\to D$ satisfying $X_{n*}\,\nu_\C = \mu_n$ and $X_n\to X$ pointwise wrt to the Scott topologies. 
\item Finally, if $X(x)\in\text{Max}\, D$, then $X_n(x)\to X(x)$ in the Lawson topology on $D$.
\end{enumerate}
\end{theorem}
\begin{proof}
Since $D$ is in \Rb, the sequence of deflations $d_k\colon D\to D$ satisfy $d_k\ll d_{k+1}$ for all $k$, and $\sup_k d_k = \text{Id}_ D$. Without loss of generality, we assume  $d_0 \equiv\, \perp_D$. Then $\Pr\, d_k = d_{k*}$  projects Prob$\, D$ onto Prob$\, d_k(D)$. Since \Pr is locally continuous, it follows that $\sup_k d_{k*} = \text{Id}_{\Pr\, D}$.          
\medbreak
 (i) To prove (i), we apply Proposition~\ref{prop:key} recursively. Let $\mu\in \Pr\, D$, and consider the sequence $\{d_{k*}\, \mu\mid k\geq 0\}$. Note that $d_{k *}\, \mu = \sum_{x\in d_k\, D} r_x\delta_x$ is simple, since $d_k$ is a deflation. We define $f_{0}\colon \C_0 \to D$ by $f_0 \equiv\, \perp_D$. Then $f_{0*}\,\nu_0 = \delta_\perp = d_{0*}\,\mu$. 
 
 For the inductive step, assume there are $m_k \geq k$ and $f_k\colon\C_{m_k}\to D$ satisfying   $f_{k*}\, \nu_{m_k} = \sum_{x\in F_k} r_x\delta_x \ll \sum_{y\in G_k} s_y\delta_y = d_{k*}\, \mu$ with $r_x\in Dyad$ for each $x\in F_k\subseteq B_D$, and 
 \begin{itemize}
  \item  $\exists \phi_k \colon F_k\to G_k\ \text{a bijection}$ with $x\ll \phi_k(x)$ and $s_{\phi_k(x)} - r_x < 2^{-m_k}$ for all $x\in F_k$.
 \end{itemize} 
 
Since $f_{k*}\,\nu_k\ll d_{k*}\,\mu\leq d_{k+1*}\,\mu$, Proposition~\ref{prop:split}(ii) implies there is a simple measure $\sum_{x\in F_{k+1}} r_x\delta_x \ll \sum_{y\in G_{k+1}} s_y\delta_y = d_{k+1*}\, \mu$ with $r_x\in Dyad$, $F_{k+1}\subseteq B_D$ and $f_{k*}\,\nu_{m_k}\ \ll \sum_{x\in F_{k+1}} r_x\delta_x$, and by making judicious choices, we may assume also satisfies:
\begin{itemize}
\item $\exists \phi_{k+1} \colon F_{k+1}\to G_{k+1}\ \text{a bijection}$ with $x\ll \phi_{k+1}(x)$ and $s_{\phi_{k+1}(x)} - r_x < 2^{-m_{k+1}}$ for all $x\in F_{k+1}$.
 \end{itemize} 
Proposition~\ref{prop:key} implies there are $m_{k+1} > k, m_k, |G_{k+1}|$ and $f_{k+1}\colon \C_{m_{k+1}}\to D$ satisfying $f_{k+1*}\,\nu_{m_{k+1}} = \sum_{y\in G_{k+1}} s_y\delta_y$ and $f_k\circ \pi_{m_k,m_{k+1}}\leq f_{k+1}$.

The resulting sequence $f_k\colon \C_{m_k}\to D$ satisfies $\{f_{k*}\,\nu_{m_k}\}$ is increasing with $\sup_k f_{k*}\,\nu_{m_k} = \sup_k d_{k*}\,\mu = \mu$ because of the bulleted items in the definition above.  

We now extend the sequence $f_k$ to an increasing sequence of Scott-continuous maps $\widehat{f_k}\colon\CT\to D$ satisfying $\widehat{f_k}_*\, \nu_\C = f_{k*}\,\nu_{m_k}$ for each $k$, which implies the supremum $\widehat{f}\colon \CT\to D$ satisfies $\widehat{f}_*\,\nu_\C = \sup_k \widehat{f_k}_*\,\nu_\C = \mu$. 

We let $\widehat{f_0} = f_0\circ \pi_0$, where $\pi_0\colon\CT\to \C_0$ is the projection. 

Assume $\widehat{f_k}\colon \CT\to D$ is defined so that $\widehat{f_k}\vert \C_{m_l} = f_l\circ \pi_{m_l,m_k}$ for each $l\leq k$. Define  $\widehat{f_{k+1}}(x) = \begin{cases} f_{k+1}\circ\pi_{k+1}(x) & \text{ if } x\in \ua\C_{m_{k+1}},\\ 
\widehat{f_k}\vert \da \C_{m_k}\circ \pi_{m_k} & \text{ otherwise.}
\end{cases}$
Then $\widehat{f_k}\leq \widehat{f_{k+1}}$ for each $k$ and $\widehat{f_k}_*\,\nu_\C = 
(f_{k}\circ \pi_k)_*\,\nu_\C = f_{k*}\,(\pi_{k*}\,\mu_\C) = f_{k*}\,\nu_{\C_{m_k}}$ for each $k$, 
so $X = \sup_k \widehat{f_k}\colon \CT\to D$ is Scott continuous and satisfies $X_*\,\mu_C = \sup_k \widehat{f_k}_*\,\nu_\C = \mu$. 
 \medbreak
 (ii) For part (ii), we appeal to Lemma~\ref{lem:cntbase} to choose a countable descending chain of Scott-open sets $U_k$ satisfying $\bigcap_k U_k = \uparrow \mu$. 
We also assume without loss of generality that $U_0 = \ua \delta_\perp = \Pr\, D$. 

Fix $n$, and note that $f_{k*}\,\nu_{\C_{m_k}}\ll \mu_n$ implies $d_{l*}\,\mu_n$ is in $U_k$ eventually in $n$. Assume the argument from part (i) has been used to construct a sequence $\widehat{g^n}_l\colon \CT\to D$ with $\sup_l \widehat{g^n}_{l*}\, \nu_\C = \mu_n$, and let $X_n = \sup_l \widehat{g^n}_l$. Then $X_n\colon \CT\to D$ is Scott continuous and $X_{n*}\,\mu_\C = \mu_n$.

It remains to show $X_n\to X$ pointwise wrt the Scott topologies. Let $x\in \CT$ with $X(x)\in U\subseteq D$ Scott open. Since $\sup_k \widehat{f_n} = X$, there is some $K$ with $\widehat{f_k}(x)\in U$ for $k\geq K$. By construction, given $k\geq K$, there is some $N'$ with $\widehat{f_k}\leq \widehat{g^n_l}\leq X_n$ for each $n\geq N'$, from which it follows that $X_n(x)\in \ua U = U$ for $n\geq N'$. 
\medbreak
(iii) For the last claim, we note that $D$ is in \RB\ implies the relative Lawson topology and the relative Scott topology agree on Max$\, D$. If $X(x)\in U\setminus\ua F$ is a Lawson open set in $D$, then there is a Scott-open set $U'\subseteq U\setminus \ua F$ with $x\in U'$. Then $X_n(x)\in U'$ eventually by part (ii), so every limit point of $\{X_n(x)\}_n$ is in $U$. Since $D$ is Lawson compact, the sequence $\{X_n(x)\}_n$ has limit points. Since $U\setminus \ua F$ is an arbitrary Lawson-open set containing $X(x)$, it follows that every limit point of $\{ X_n(x)\}_n$ is in $\bigcap_{X(x)\in U\setminus \ua F} U = \ua X(x) = \{X(x)\}$, the last equality following from the fact that $X(x)$ is maximal. Thus, $X_n(x)\to X(x)$ in the Lawson topology.
\end{proof}
The next result follows from Theorem~\ref{thm:mapping} by simply considering the law $\mu = X_*\,\nu_\C$. 
\begin{corollary}\label{cor:random}
If $D$ is a countably based \RB\ domain and $X\colon \C \to D$ is a random variable, then there is a Scott-continuous map $f\colon \CT\to D$ satisfying $f_*\,\nu_\C = X_*\,\nu_\C$. 
\end{corollary}

\begin{remark}
Theorem~\ref{thm:mapping} still holds if we weaken the hypothesis to assuming only that $D$ is a countably based coherent domain. The only change in the proof is that the simple measures $d_{k *}\mu$ approximating each measure $\mu$ must be chosen measure-by-measure, since coherent domains don't have a sequence of deflations such as the $d_k$s that are available for \RB-domains. But having a countable basis ensures that each measure $\mu$ has a countable sequence of simple measures $\sigma_n$ with $\mu = \sup_n \sigma_n$ and that also satisfy $\sigma_n\ll \sigma_{n+1}\ll \mu$. One can use these in place of the ``uniformly chosen" simple measures $d_{k *}\,\mu$. The remainder of the proof proceeds along the same lines, with the simple measures with dyadic coefficients chosen by recursion. 
\end{remark}

The obvious question is whether Theorem~\ref{thm:mapping} extends to subprobability measures. In fact, it does, as we now outline. Given an \RB domain $D$, then  $D_\perp$, the lift of $D$, also is in \textsf{RB}, and it has a countable basis if $D$ does. Then, the embedding $D\hookrightarrow D_\perp$ allows us to define an embedding $e\colon \textsf{SProb}\, D\to \textsf{Prob}\, D_\perp$ by $e(\mu) = \mu + (1 - \mu(D))\delta_\perp$. Theorem~\ref{thm:mapping} then implies there is a Scott-continuous map $f\colon \CT\to D_\perp$ with $f_*\, \nu_\C = e(\mu)$. Note that $f^{-1}(\{\perp\}) = C_f$ is Scott closed, and the restriction $f' \equiv\,f\vert_{D\setminus \{\perp\}}\colon \CT\setminus C_f\to D$ also is Scott continuous, and it's routine to show that $f'_*\, \nu_\C = \mu$. Likewise, any sequence of subprobability measures $\mu_n\to_w \mu$ converging weakly in \textsf{SProb}$\, D$ satisfies the property that there are Scott-continuous partial maps $f'_n\colon \CT\setminus C_n\to D$ with $f'_{n *}\, \nu_\C = \mu_n$. Theorem~\ref{thm:mapping}(ii) shows $f_n \to f$ in the Scott topologies, and this in turn implies that  $f'_n \to f'$ in the Scott topologies. We summarize this as follows:

\begin{corollary}\label{cor:mapping}
Let $D$ be a countably based \RB domain. Then:
\begin{itemize}
\item If $\mu\in \textsf{SProb}\, D$, there is a Scott-continuous map $f\colon \CT\setminus C_f\to D$ satisfying $f_*\, \nu_\C = \mu$, where $C_f$ is Scott closed.
\item If $\mu_n\to_w \mu$ is a sequence of subprobability measures on $D$ converging weakly to the subprobability measure $\mu$, then there is a sequence of of Scott-continuous maps $f_n\colon \CT\setminus C_n\to D$ satisfying $f_{n *}\, \nu_\C = \mu_n$, , where $C_n$ is Scott closed, for each $n$, and $f_n \to f$ in the Scott topologies. 
\end{itemize}
\end{corollary}

\section{Domains, Polish spaces and Skorohod's Theorem}\label{sec:polish}
In this section we present the principal application of our main results. We begin with the necessary background about Polish spaces and random variables, and  then outline a representation theorem for Polish spaces topologically embedded as $G_\delta$-subsets of a domain in the relative Scott topology. With this in place, we focus on probability measures and derive  a domain-theoretic proof of Skorohod's Theorem. 
\begin{definition}
A \emph{Polish space} is a completely metrizable separable topological space. I.e., $P$ is Polish iff $P$ is homeomorphic to a complete metric space that has a countable dense subset. 
\end{definition}

Polish spaces figure prominently in probability theory~\cite{billings}, as well as in descriptive set theory~\cite{kechr}. As we commented in the last section, one approach to probability theory~\cite{billings} begins with metric spaces, and Polish spaces are where the deepest results hold. Since our results all involve separable topological spaces, the measurable sets in all cases are the Borel sets. So, we use the term Borel set instead of measurable set.

The appropriate mappings in probability theory are \emph{random variables} -- measurable maps $X\colon (P,\Sigma_P, \mu)\to (S,\Sigma_S)$, where $(P,\Sigma_P,\mu)$ is a probability space, 
$(S,\Sigma_S)$ is a measure space, and $X^{-1}(A)\in\Sigma_P$ for each $A\in \Sigma_S$. If 
$X\colon P\to S$ is a random variable, then the \emph{push forward} of $\mu$ by $X$ is the 
measure $X_*\,\mu$ defined by $X\,\mu (A) = \mu(X^{-1}(A))$ for all measurable sets $A\subseteq 
S$. Equivalently, a function $f\colon S\to \Re$ is $X\,\mu$-integrable iff $f\circ X$ is 
$\mu$-integrable, and in this case, $\int fd X_*\,\mu = \int f\circ X d\mu$. We now show how our results 
from Section~\ref{sec:cantor} can be applied to Polish spaces that can be topologically embedded as 
the maximal elements of a domain in the relative Scott topology.  We begin by extending some 
results about probability measures on Polish spaces to sub-probability measures. 

\begin{definition} Let $\mu$ be a probability measure on a space $X$. A Borel set $A\subseteq X$ is a \emph{$\mu$-continuity set} if $\mu(\overline{A}\setminus A) = 0$. Since $\overline{A}$ is closed and $A$ is Borel, $\overline{A}\setminus A$ is Borel.
\end{definition}

\begin{theorem}\label{thm:port}(Portmanteau Theorem)
Let $X$ be a Polish space, and let $\textsf{SProb}\, X$ denote the family of sub-probability measures on $X$ in the weak topology, and let $\mu_n,\mu\in\textsf{SProb}\, X$. Then the following are equivalent:
\begin{enumerate}
\item $\mu_n\to_w \mu$ in the weak topology.
\item $\int f d\mu_n \to \int f d\mu$ for all bounded uniformly continuous $f\colon X\to \Re$. 
\item $\liminf_n \mu_n(O)\geq \mu(O)$ for all open sets $O\subseteq X$.
\item $\limsup_n \mu_n(F)\leq \mu(F)$ for all closed sets $F\subseteq X$.
\item $\mu_n(A) \to \mu(A)$ for all $\mu$-continuity sets $A$. 
\end{enumerate}
\end{theorem}
\begin{proof}
Theorem 2.1 of~\cite{billings} shows these conditions are equivalent if $\mu_n,\mu$ are probability measures on a metric space. We assume the metric $d$ on $X$ satisfies $\text{diam}\, X < 1$ by normalization, if necessary, and then create a new space $X' = X\stackrel{\cdot}{\cup} \{*\}$, where $*$ is an element not in $X$. If $d$ is the metric on $X$, we extend $d$ to $X'$ by setting $d(*,x) = d(x,*) = 1$ for all $x\in X$. This makes $X'$ into a Polish space, and there is an embedding $e\colon \textsf{SProb}\, X \hookrightarrow \textsf{Prob}\, X'$ by $e(\mu) = \mu + (1 - \mu(X))\delta_*$. Then the conditions (i) -- (v) are equivalent for $e(\textsf{SProb}\, X)$.  But $\mu = e(\mu)\vert_X$ for all $\mu\in \text{SProb}\, X$, and $X$ is clopen in $X'$, so they also are equivalent for $\textsf{SProb}\, X$. 
\end{proof}

For the next result, recall that a measure $\mu$ on a space $X$ is \emph{concentrated on the Borel set $A\subseteq X$} if $\mu(X\setminus A) = 0$. 
\begin{proposition}\label{prop:prob-embed}
Let $X$ be a Polish space with a topological embedding $\iota\colon X\to D$  as a $G_\delta$-subset of a countably-based domain, $D$. Define $e\colon \textsf{SProb}\, X\to \textsf{SProb}\, D$ by $e(\mu) = \iota_*\,\mu$. Then:
\begin{enumerate}
\item $e$ is one-to-one.
\item $e(\textsf{SProb}\, X) = \{ \mu\in \textsf{SProb}\, D\mid \mu \text{ concentrated on } \iota(X)\}$, and $j\colon e(\textsf{SProb}\, X)\to \textsf{SProb}\, X$ by $j(\nu) = \nu \circ \iota$ is inverse to $e$. 
\item $e(\textsf{Prob}\, X) \subseteq \textsf{Prob}\, D =\ $\emph{\text{Max}}$\, \textsf{SProb}\, D$ is a Borel subset of $\textsf{SProb}\, D$. 
\end{enumerate}
\end{proposition}
\begin{proof} The same result for $\textsf{Prob}\,X$ is Proposition 4.2 of ~\cite{weaktop}. That proof relies on Proposition 4.1 of the same paper, which characterizes properties of the embedding of $X$ into $D$, and hence applies equally to sub-probability measures.  With this result in hand, the proofs in~\cite{weaktop} apply almost verbatim for sub-probability measures, except that part (ii) uses $\mu(X) = 1$, but this can be replaced by $\mu(X) = ||\mu||$. The final point that $e(\textsf{SProb}\, X)$ is a Borel set follows from Lemma 2.3 of~\cite{varad}.
\end{proof} 

\begin{theorem}\label{thm:embedcont}
Let $X$ be a Polish space and $\iota\colon X\to D$ a topological embedding of $X$ into a countably based domain $D$. Then the mapping $e\colon \textsf{SProb}\, X\to \textsf{SProb}\, D$ by $e(\mu) = \iota_*\, \mu$ is a topological embedding relative to the weak topologies on $\textsf{SProb}\, X$ and on $\textsf{SProb}\, D$. 
\end{theorem}
\begin{proof}
We first show $e\colon \textsf{SProb}\, X\to \textsf{SProb}\, D$ is continuous: Let $\mu_n,\mu\in\textsf{SProb}\, X$ with $\mu_n\to_w \mu$. To show $e(\mu_n)\to_w e(\mu)$, let $O\subseteq D$ be Lawson open. Then:
\begin{eqnarray*}
\liminf_n e(\mu_n)(O) &= & \liminf_n \iota\, \mu_n(O)\\ &=& \liminf_n \mu_n(\iota^{-1}(O)) \notag\\
& \geq & \mu(\iota^{-1}(O)) \qquad\qquad \text{by Theorem~\ref{thm:port}(ii)}\\
&=& e(\mu)(O).
\end{eqnarray*}
Then Theorem~\ref{thm:topvsweak}(ii) implies $e(\mu_n)\to_w e(\mu)$. 

For the converse, let $\nu_n\to_w \nu$ in $\textsf{SProb}\, D$ with $\nu_n, \nu$ all concentrated on $\ran e$. Since $\iota\colon X\to D$ is a topological embedding, given an open set $O\subseteq X$, there is some Lawson open $O'\subseteq D$ satisfying $\iota(O) = O'\cap \iota(X)$. If $j\colon \textsf{SProb}\,D\to \textsf{SProb}\, X$ is the restriction map, then
\begin{eqnarray*}
\liminf_n j(\nu_n)(O) & = & \liminf_n \nu_n(\iota(O))\\  
& = & \liminf_n \nu_n(O'\cap \iota(X))\\
& = & \liminf_n \nu_n(O') \qquad\quad\ \text{$\nu_n$ is concentrated on $\iota(X)$}\\
&\geq & \nu(O') \qquad\qquad\qquad\ \text{ by Theorem~\ref{thm:topvsweak}(ii)}\\
&=& j(\nu)(O).
\end{eqnarray*}
It follows by Theorem~\ref{thm:port}(ii) that $j(\nu_n)\to_w j(\nu)$ in \textsf{SProb}$\, X$.
\end{proof}
\begin{remark}
Theorem~\ref{thm:embedcont} is Corollary 4.1 of~\cite{weaktop} extended to the case of sub-probability measures, from the case of probability measures. The proof is identical to the one in~\cite{weaktop}, except the reasoning has been changed to rely on the results we established for \textsf{SProb}.
\end{remark}

\subsection{Bounded complete domains and Skorohod's Theorem}
The connection between domains and Polish spaces involves computational models. A \emph{computational model for a topological space $X$} is a domain $D$ for which there is a topological embedding $X\simeq \text{Max}\, D$ of $X$ as the space of maximal elements of $D$ endowed with the relative Scott topology. As described in the Introduction, this notion emerged from the work of Edalat, who developed the first domain models of spaces arising in real analysis using the domain of compact subsets of the space under reverse inclusion. Later, Lawson~\cite{lawson} showed that the space $\text{Max}\, D$ of maximal elements of a bounded complete countably based domain in the relative Scott topology is a Polish space, and Ciesielski, Flagg and Kopperman~\cite{kopperman,kopperman2} showed that every Polish space has such a model. Finally, Martin~\cite{martin} noted that the space of maximal elements of any countably based, bounded complete domain is a $G_\delta$ in the relative Scott topology. Since these results play a crucial role in our work, we state them formally:

\begin{theorem}\label{thm:polish}(Lawson~\cite{lawson}, Ciesielski, et al.~\cite{kopperman,kopperman2}, Martin~\cite{martin}) 
A space $X$ is representable as $\text{Max}\, D$ in the relative Scott topology, for $D$ a countably based, bounded complete domain, iff $X$ is a Polish space. In such a representation, $X$ is a $G_\delta$--subspace of $D$ in the Scott topology. 
\end{theorem}

Our goal is to prove prove Skorohod's Theorem using our results from Section 3. But before we do that, we need one more preparatory result. W.\ M.\ Schmidt was the first to observe that the canonical surjection $\pi\colon C\to [0,1]$ from the Cantor set, $C\simeq 2^\Nat$ sends Haar measure $\nu_C$ to Lebesgue measure. We need the lower adjoint of that projection for our proof.

\begin{proposition}\label{prop:lebesgue} 
Let $C\simeq 2^\Nat$ denote the Cantor set regarded as the countable product of two-point groups. Then the canonical projection $\pi\colon C\to [0,1]$ has a lower adjoint $j\colon [0,1]$ satisfying $j([0,1]) = C\setminus KC$, and $j\, \lambda = \nu_C$,  where $\lambda$ denotes Lebesgue measure.
\end{proposition}
\begin{proof}
For a proof that $\pi\, \nu_C = \lambda$, see~\cite{brian-misl}, which also has an extensive discussion of related results, including the following. First, $\pi$ preserves all sups and all infs, so, in particular, it has a lower adjoint $j\colon [0,1]\to C$ preserving all suprema. Then $j([0,1]) = C\setminus KC$, where $KC$ is the set of compact elements, which is countable, so $\nu_C(KC) = 0$. The pair of maps, $\pi\vert_{C\setminus KC}$ and $j$, form a Borel isomorphism. 

We claim $j\, \lambda = \nu_C$: Indeed, if $A\subseteq C$ is a Borel set, then 
\begin{eqnarray*}
j\, \lambda(A) = \lambda(j^{-1}(A)) &=& \lambda(j^{-1}(A\setminus KC))\\ &=&\pi\, \nu_C(j^{-1}(A\setminus KC))\\ &=& \nu_C( \pi^{-1}\circ j^{-1} (A\setminus KC))\\ &=&  \nu_C(A\setminus KC) = \nu_C(A),
\end{eqnarray*} 
since $\nu_C(KC) = 0$.
\end{proof}

A \emph{stochastic process on a measure space $(S,\Sigma_S)$} is a family $\{ X_t\mid t\in T\subseteq \Re+\}$ of random variables $X_t\colon (P,\Sigma_P,\nu)\to (S,\Sigma_S)$, where $(P,\Sigma_P, \nu)$  is a probability space. It's often assumed that $(S,\Sigma_S)$ is a Polish space, in which case $\textsf{Prob}\, S$ also is Polish in the weak topology. The push forward measure $X_{t *}\, \nu\in \textsf{Prob}\, S$ is called the \emph{law of $X_t$}, and  a natural question is the convergence properties of the family $\{X_{t *}\, \nu\mid t\in T\}\subseteq \textsf{Prob}\, S$. Since $\textsf{Prob}\, S$ is Polish, convergence can be defined using sequences, and it's obvious that if $X_{t_n}\to X_t$ a.e. on $X$, then $X_{t_n *}\, \nu \to_w X_{t *}\, \nu$ in $\textsf{Prob}\, S$. Skorohod's Theorem not only provides a converse to this observation, it also shows the probability space $(P,\Sigma_P,\nu)$ can be assumed to be the unit interval with Lebesgue measure, $\lambda$:

\begin{theorem}\label{thm:skor} (Skorohod's Theorem~\cite{skorohod}) 
If $P$ is a Polish space and $\mu\in\textsf{Prob}\, P$, then there is a random variable $X\colon [0,1]\to P$ satisfying $X_*\,\lambda = \mu$. 

Moreover, if $\mu_n, \mu\in \textsf{Prob}\,P$ satisfy $\mu_n\to_w \mu$ in the weak topology, then the random variables $X,X_n\colon [0,1]\to P$ can be chosen so that $X_*\,\lambda = \mu$ and $X_{n *}\,\lambda = \mu_n$ also satisfy $X_n\to X$ a.s.~wrt Lebesgue measure. 
\end{theorem} 

\begin{proof} 
Since $P$ is Polish, Theorem~\ref{thm:polish} implies there is a bounded complete domain $D_P$ with a countable base and an embedding $e\colon P\to \text{Max}\, D$ and $e(P)$ is a $G_\delta$ subset of $D_P$. Then $\textsf{Prob}\, e = e_*\colon \textsf{Prob}\, P\to \textsf{Prob}\, D_P$ is an embedding. Theorem~\ref{thm:mapping} then implies there are Scott-continuous maps $X',X_n'\colon \CT\to D_P$ satisfying $X_*'\, \nu_\C = \mu, X_{n *}'\, \nu_\C = \mu_n$, and $X_n'\to X'$ is the Scott topologies. But since $\mu, \mu_n\in \text{Max}\, D_P$, it follows that $X_n'\to X'$ in the Lawson topologies. 

If we restrict the mappings $X', X_n'$ to those $x\in \text{Max}\, \CT$ with $X'(x),X_n'(x)\in \text{Max}\,{D}$, and then precompose with $j$, we then have random variables $X_n = X_n' \circ j, X = j\circ X'\colon [0,1]\to P$ as desired: $X_{n *}\, \lambda = X_{n *}'\, \nu_C = \mu_n, X_*\, \lambda = X_*'\, \nu_C = \mu$, and $X_n\to X$ a.s. on $[0,1]$ wrt $\lambda$. 
\end{proof}

\begin{remark} A \emph{standard Borel space} is a measurable space for which there is a Borel isomorphism onto a Borel subset of a Polish space~\cite{kechr}. An obvious example is the unit interval. A \emph{standard probability space} is then a probability space $P,\Sigma_P,\nu)$ that is 
isomorphic mod$\, 0$ to the unit interval equipped with Lebesgue measure, where ``isomorphic mod$\,0$" means there are Borel sets $A\subseteq P$ and $B\subseteq [0,1]$ both of measure 0 so that $P\setminus A$ is Borel isomorphic to $[0,1]\setminus B$.  This leads to two comments:
\begin{enumerate}
\item The thrust of Skorohod's Theorem is that the domain for any stochastic process whose range is a Polish space can be assumed to be a standard probability space. Some statements of the theorem simply state it that way, without specifying which standard probability space is being used. But most often, the standard space is assumed to be the unit interval with Lebesgue measure.
\item  Any two standard Borel spaces that are uncountable are Borel isomorphic (cf.~\cite{kechr}). Clearly $\C$ is such a space, as is the canonical example, $[0,1]$. So we could have used $\C$ in Theorem~\ref{thm:skor} instead of $[0,1]$.
\end{enumerate}
\end{remark}

Finally, just as with Theorem~\ref{thm:mapping}, on which Theorem~\ref{thm:skor} relies, this result extends \emph{verbatim} to subprobability measures. T
\begin{corollary}\label{cor:skor}
If $P$ is a Polish space and $\mu$ is a subprobability measure on $P$, then there is a random variable $X\colon (0,1]\to P$ satisfying $X_*\,\lambda = \mu$. 

Moreover, if $\mu_n, \mu\in \textsf{SProb}\,P$ satisfy $\mu_n\to_w \mu$ in the weak topology, then the random variables $X,X_n\colon (0,1]\to P$ can be chosen so that $X_*\,\lambda = \mu$ and $X_{n *}\,\lambda = \mu_n$ also satisfy $X_n\to X$ a.s.~wrt Lebesgue measure. 
\end{corollary}
\begin{proof}
Apply Corollary~\ref{cor:mapping}, and precede the mappings defined there with the lower semicontinuous embedding $j\colon (0,1] \to \C$. 
\end{proof}

\section{Summary and Future Work}
In this paper we have developed a domain-theoretic approach to random variables. Our main results establish some standard results in probability theory using the Cantor tree as a domain, in which measurable maps from the Cantor set are approximated by Scott-continuous maps defined on the tree. We have recast Skorohod's Theorem in domain-theoretic terms, and we've also extended the theorem to subprobability measures using our approach, both for domains and for the classic case of Polish spaces. We also presented a direct proof that the sub-probability measures and the probability measures on a complete chain form a continuous lattice. This result offers the first new insight in over two decades to the domain structure of $\SPr{D}$ and $\Pr{D}$, the last such results having appeared in~\cite{jungtix}.

There are a number of interesting questions to be explored. Finding further results from random variables that can be obtained using the techniques presented here is an obvious issue. Another question we have been exploring is the potential use of the disintegration theory for product measures, in order to understand the domain structure of $\SPr{(D\times E)}$, in the case $D$ and $E$ are chains. In particular, we'd like to know if we can use the fact that  $\SPr{D}$ and $\SPr{E}$ are continuous lattices to derive some insight into the domain structure of $\SPr{(D\times E)}$. Our first idea -- that this family of measures also would be a lattice -- is debunked by a simple example in~\cite{jones}, so more subtle issues are at play here. Last, we're interested in investigating the potential application of our ideas to probabilistic programming semantics and Bayesian inference.
\begin{acknowledgement}
This work began while the author was a participant in the Logical Methods in Computation Program at the Simons Institute for the Theory of Computing in Fall, 2016. The author thanks the institute and its staff for providing an excellent environment for research.  
The author also wishes to acknowledge the support of the US AFOSR during the preparation of this work.
\end{acknowledgement}
\bibliographystyle{plain}

\end{document}